\documentclass[11pt,reqno,a4paper]{amsart}
\usepackage{amsmath,amsthm,verbatim,amscd,amssymb,setspace,hyperref,esint,mathtools}
\usepackage{exscale,color}

\usepackage{mathrsfs}

\usepackage{enumitem}

\usepackage[colorinlistoftodos,prependcaption]{todonotes}

\renewcommand{\epsilon}{\varepsilon}

\newcommand{\R}{\mathbb{R}}

\newcounter{mtheorem}
\newtheorem{mtheorem}[mtheorem]{Theorem}

\newcommand{{\vol}}{\rm vol}

\newcommand{\Ric}{\operatorname{Ric}}

\def\Ric{\operatorname{Ric}}

\def\vol{\operatorname{vol}}

\newtheoremstyle{fancy}{}{}{\itshape}{}{\textbf\bgroup}{.\egroup}{ }{}
\newtheoremstyle{fancy2}{}{}{\rm}{}{\textbf\bgroup}{.\egroup}{ }{}

\theoremstyle{fancy}
\newtheorem{theorem}{Theorem}[section]
\newtheorem{lemma}[theorem]{Lemma}
\newtheorem{corollary}[theorem]{Corollary}

\newtheorem{prop}[theorem]{Proposition}

\theoremstyle{fancy2}
\newtheorem{definition}[theorem]{Definition}

\newtheorem{remark}[theorem]{Remark}

\newtheorem{claim}[theorem]{Claim}

\setlist{leftmargin=*}

\numberwithin{equation}{section}
\begin{document}
\title{Uniqueness of the soliton vector field of a shrinking gradient K\"ahler-Ricci soliton}
\date{\today}

\author{Ronan J.~Conlon}
\address{Department of Mathematical Sciences, The University of Texas at Dallas, Richardson, TX 75080}
\email{ronan.conlon@utdallas.edu}
\author{Alix Deruelle}
\address{Universit\'e Paris-Saclay, CNRS, Laboratoire de math\'ematiques d'Orsay, 91405, Orsay, France}
\email{alix.deruelle@universite-paris-saclay.fr}

\begin{abstract}
We establish the uniqueness of the soliton vector field of a complete shrinking gradient K\"ahler-Ricci soliton up to pushforward by biholomorphism.
\end{abstract}

\maketitle

\markboth{Ronan J.~Conlon and Alix Deruelle}{Uniqueness of the soliton vector field of a shrinking gradient K\"ahler-Ricci soliton}

\section{Introduction}

\subsection{Overview}
A \emph{Ricci soliton} is a triple $(M,\,g,\,X)$, where $M$ is a Riemannian manifold endowed with a complete Riemannian metric $g$
and a complete vector field $X$, such that
\begin{equation}\label{hot}
\Ric(g)+\frac{1}{2}\mathcal{L}_{X}g=\lambda g
\end{equation}
for some $\lambda\in\mathbb{R}$. The vector field $X$ is called the
\emph{soliton vector field}. If $X=\nabla^{g} f$ for some smooth real-valued function $f$ on $M$,
then we say that $(M,\,g,\,X)$ is \emph{gradient}. In this case, the soliton equation \eqref{hot}
becomes $$\Ric(g)+\operatorname{Hess}_{g}(f)=\lambda g,$$
and we call $f$ the \emph{soliton potential}. In the case of gradient Ricci solitons, the completeness of $X$ is guaranteed by the completeness of $g$
\cite{zhang12}.

Let $(M,\,g,\,X)$ be a Ricci soliton. If $g$ is K\"ahler with K\"ahler form $\omega$ and $X$ is real holomorphic, then we say that $(M,\,g,\,X)$ 
(or $(M,\,\omega,\,X)$) is a \emph{K\"ahler-Ricci soliton}. If $(M,\,g,\,X)$ is in addition gradient, then \eqref{hot} may be rewritten as
\begin{equation}\label{krseqn}
\rho_{\omega}+i\partial\bar{\partial}f=\lambda\omega,
\end{equation}
where $\rho_{\omega}$ is the Ricci form of $\omega$ and $f$ is the soliton potential.

Finally, a Ricci soliton and a K\"ahler-Ricci soliton are called \emph{steady} if $\lambda=0$, \emph{expanding}
if $\lambda<0$, and \emph{shrinking} if $\lambda>0$ in \eqref{hot}.
One can always normalise $\lambda$, when non-zero, to satisfy $|\lambda|=1$. We henceforth assume that this is the case.

Ricci solitons are interesting both from the point of view of canonical metrics and of the Ricci flow. On one hand, they represent one direction in which the notion of an Einstein manifold can be generalised. On compact manifolds, shrinking Ricci solitons are known to exist in several instances where there are obstructions to the existence of Einstein metrics; see, for example, \cite{soliton}. By the maximum principle, there are no non-trivial expanding or steady Ricci solitons on compact manifolds. However, there are many examples on non-compact manifolds; see, for example, \cite{Cao-KR-sol, FIK, Wang, con-der, conlon33, schafer2, schaffer, heather}. On the other hand, one can associate to a Ricci soliton a self-similar solution of the Ricci flow, and gradient shrinking Ricci solitons in particular provide models for finite-time Type I singularities of the flow \cite{naber, topping}. From this perspective, it is important to 
study such solitons in order to better understand singularity formation in the Ricci flow.

In this article, we prove a uniqueness result for the soliton vector field of a shrinking gradient K\"ahler-Ricci soliton. 

\subsection{Main results}
Our main result is the following, which provides an answer to \cite[Question (7), Section 7.2]{cds}.
\begin{mtheorem}[Uniqueness of the soliton vector field for shrinkers]\label{uunique}
Let $M$ be a non-compact complex manifold. Then up to pushforward by biholomorphism, there
exists at most one real holomorphic vector field $X$ on $M$ that admits a complete shrinking gradient K\"ahler-Ricci soliton $(M,\,g,\,X)$ with $X=\nabla^g f$ for a smooth real-valued function $f$ on $M$.
\end{mtheorem}

This result was proved in the compact case by Tian-Zhu \cite{tianzhu2}. In the non-compact case, uniqueness was proved
under the assumption of bounded Ricci curvature for soliton vector fields lying in the 
Lie algebra of a given real torus \cite[Theorem D]{cds}. The torus assumption was subsequently removed in \cite[Theorem 1.3]{Esp25}.
For toric shrinking gradient K\"ahler-Ricci solitons,
uniqueness was proved without a curvature bound in \cite[Theorem 4.6]{charlie}.

The proof of Theorem \ref{uunique} follows closely the argument for the corresponding result in the bounded Ricci curvature case; cf.~\cite[Theorem D]{cds}. 
In essence, we demonstrate that the Ricci curvature bound can be removed because the zero set of the soliton vector field $X$ is now known to be compact \cite[Proposition 3.5]{JunshengSong}. More precisely, there are three essential inputs for the proof of Theorem \ref{uunique}.
\begin{itemize}
  \item  The first is Theorem \ref{matty}, a Matsushima-type theorem for complete shrinking gradient K\"ahler-Ricci solitons. The fact that the zero set of the soliton vector field is compact yields the weighted $L^{2}$-bound of Proposition \ref{keywest} for real vector fields commuting with the soliton vector field. With this, the proof of \cite[Theorem 5.1]{cds} carries over to give the desired Matsushima theorem. 
  \item  The second main ingredient is the characterisation of $JX$ as the minimiser of the weighted volume functional in Lemma \ref{critical}, the new idea being that we exploit the properness of the potentials themselves.
  \item The third is Iwasawa’s theorem for maximal tori in the automorphism
group of a polarised Fano fibration that contain a Reeb vector field \cite[Proposition 5.17]{Esp25}.
  \end{itemize}
With these three ingredients in hand, the proof of \cite[Theorem D]{cds} carries over almost verbatim to give Theorem \ref{uunique}.

\subsection{Acknowledgements}
The first author is supported by a Simons Travel Grant. The second author is partially supported by grants from the
French National Research Agency ANR-24-CE40-0702 (Project OrbiScaR) and the Fondation Charles Defforey-Institut de France via the project
``KRIS''. He also benefits from a Junior Chair from the Institut Universitaire de France.

\subsection{Declaration on the Use of AI} 
The authors used ChatGPT 6-Astra to assist with literature review, checking mathematical arguments, and improving the language and presentation. The authors retain exclusive responsibility for the correctness, accuracy, and validity of the results presented in this work.

\subsection{Notation}
For clarity, we provide a summary of our notation. 
\begin{itemize}
 \item $(M,\,g,\,X)$ is a shrinking gradient K\"ahler-Ricci soliton on a complex manifold $M$ with K\"ahler metric $g$ and soliton vector field $X$.
\item $n$ is the complex dimension of $M$.
\item $\omega$ is the K\"ahler form of $g$.
 \item $J$ is the complex structure of $M$.
 \item $X=\nabla^{g}f$ for a smooth real-valued function $f:M\to\mathbb{R}$, the soliton potential.
 \item $L^{2}(e^{-f}\omega^{n})=\{u\in C^{\infty}(M)\,|\,\int_{M}|u|^{2}\,e^{-f}\omega^{n}<\infty\}$.
\item $\mathfrak{g}^{X}$ is the Lie algebra of real holomorphic Killing vector fields that commute with $X$. 
 \item $G_{0}^{X}$ is the connected component of the identity of the Lie group of holomorphic isometries
of $(M,\,J, \,g)$ that commute with the flow of $X$ and has Lie algebra $\mathfrak{g}^{X}$; cf.~Lemma \ref{class}.
\item $\mathfrak{aut}^{X}(M)$ is the Lie algebra of real holomorphic vector fields that commute with $X$.
\item $\operatorname{Aut}_{0}^{X}(M)$ is the connected component of the identity of the Lie group of biholomorphisms of $M$ that commute with the flow of $X$ 
and has Lie algebra $\mathfrak{aut}^{X}(M)$; cf.~Proposition \ref{eastbeach}.
\item $\operatorname{Aut}(M)$ is the automorphism group of $M$.
\end{itemize}

\section{Preliminaries}\label{sec_preliminaries}

In this section, we consider some properties of shrinking gradient Ricci and K\"ahler-Ricci solitons. 
\subsection{Properties of shrinking Ricci solitons}

We have the following:
\begin{theorem}\label{prop-sum-sgrs}
Let $(M^{m},\,g,\,X)$ be a complete non-compact shrinking gradient Ricci soliton of real dimension $m$ 
satisfying $\Ric(g)+\frac{1}{2}\mathcal{L}_{X}g=g$ with scalar curvature $R_{g}$ and soliton vector field $X=\nabla^{g}f$ for a smooth real-valued function $f:M\to\mathbb{R}$.
Then the following hold true.
\begin{enumerate}
\item \textnormal{(Non-negativity of the scalar curvature {\cite[Theorem 1.3(ii)]{zhang12}}).}~$R_{g}\geq0$. \\

\item \textnormal{(Growth of the soliton potential {\cite[Theorem 1.1]{caoo}}).}~For $x\in M$, $f$ satisfies the estimates
$$\frac{1}{2}(d_{g}(p,\,x)-c_{1})^{2}-C\leq f(x)\leq\frac{1}{2}(d_{g}(p,\,x)+c_{2})^{2}$$
for some $C>0$, where $d_{g}(p,\,\cdot)$ denotes the distance to a fixed point $p\in M$ with respect to $g$.
Here, $c_{1}$ and $c_{2}$ are positive constants depending only on the real dimension of $M$ and
the geometry of $g$ on the unit ball $B_{p}(1)$ based at $p$. \\
\item \textnormal{(Polynomial volume growth \cite[Theorem 1.2]{caoo}).}~For each $x\in M$, there exists a positive constant $C>0$
such that $$\operatorname{vol}_{g}(B_{r}(x))\leq Cr^{m}\quad\textrm{for $r > 0$ sufficiently large}.$$\\
\end{enumerate}
In particular, $\int_{M} e^{-f}d\mu<\infty$, where $d\mu$ is the volume form of $g$.
\end{theorem}

The next lemma collects together some well-known identities for shrinking gradient K\"ahler-Ricci solitons that we will exploit. Their proofs 
can be found for instance in \cite[Chapter 4]{Cho-Lu-Ni-Boo}. 
\begin{lemma}[Ricci soliton identities]\label{solitonid}
Let $(M^{n},\,g,\,X)$ be a shrinking gradient K\"ahler-Ricci soliton of complex dimension $n$ satisfying \eqref{krseqn} with $\lambda=1$
with soliton vector field $X=\nabla^{g}f$ for a smooth real-valued function $f:M\to\mathbb{R}$. Then the trace and first order soliton identities are:
\begin{equation*}
\begin{split}
&\Delta_{\omega} f +\frac{R_{g}}{2}=n,\\
&\nabla^g R_{g}-2 \Ric(g)(X)=0,\\
&|\nabla^g f|^2+R_{g}-2f=\operatorname{const.},\\
\end{split}
\end{equation*}
where $R_{g}$ denotes the scalar curvature of $g$ and $|\nabla^g f|^2:=g^{ij}\partial_if\partial_{j}f$.
\end{lemma}

\begin{remark}\label{rk-sol-id}
We henceforth normalise the soliton potential $f$ of a shrinking gradient K\"ahler-Ricci soliton of complex dimension $n$ satisfying \eqref{krseqn} with $\lambda=1$ so that 
\begin{equation}\label{pressure}
|\nabla^g f|^2+R_{g}-2f=2n. 
\end{equation}
The choice of constant $2n$ is dictated by the following equation satisfied by $f$:
\begin{equation}\label{normal-f}
\Delta_{\omega}f-\frac{X}{2}\cdot f=-f.
\end{equation}
This choice of constant also implies that $f+n$ is non-negative on $M$ since the scalar curvature $R_g$ of $g$ is necessarily non-negative.
\end{remark}

\subsection{Properties of real holomorphic vector fields commuting with the soliton vector field}
In this section, we mention some properties of real vector fields that commute with the soliton vector field
on a complete shrinking gradient K\"ahler-Ricci soliton. Many of these properties were proven in \cite{cds} under the assumption of bounded Ricci curvature. Here we show that, 
thanks to the fact that the zero set of the soliton vector field is now known to be compact \cite{JunshengSong}, the Ricci curvature bound is unnecessary.

We begin with the weighted integrability for real vector fields commuting with the soliton vector field.
\begin{prop}[\protect{\textnormal{cf.~\cite[Proposition 2.30]{cds}}}]\label{keywest}
Let $(M,\,g,\,X)$ be a complete shrinking gradient K\"ahler-Ricci
soliton with K\"ahler form $\omega$ and soliton vector field
$X=\nabla^{g}f$ for a smooth real-valued function
$f:M\to\mathbb{R}$.
Then every smooth real vector field $Y$ on $M$ with
$[X,\,Y]=0$ lies in $L^2(e^{-f}\omega^n)$.
\end{prop}

\begin{proof}
Let $M_{0}(X)$ denote the zero set of $X$ and choose
$A>\max\left\{1,\,\sup_{M_{0}(X)}f\right\}$.
This we can do because $M_{0}(X)$ is compact by
\cite[Proposition 3.5]{JunshengSong}. Let $\Phi_t$ denote the flow of $X$ for time $t$.
Since $X$ is complete \cite{zhang12}, $\Phi_t$ is defined
for all time.
Let $\Sigma:=f^{-1}(3A)$,
$N:=f^{-1}((-\infty,\,3A])$, and
\[
K:=\{\Phi_s(x)\,|\,x\in\Sigma,\ 0\leq s\leq1\}.
\]
Since $f$ is proper and bounded below as a consequence of
Theorem \ref{prop-sum-sgrs}(ii), $\Sigma$ and $N$ are compact
subsets of $M$, and $M_{0}(X)\subset N$. $K$ is clearly compact as the image of $[0,1]\times\Sigma$
under a continuous map.
Moreover, since $f$ is non-decreasing along the flow of $X$,
we have that $f\geq3A$ on $K$.

Write $\gamma_x(t):=\Phi_t(x)$ for the integral curve of
$X$ with $\gamma_x(0)=x\in M$.
We begin with the following claim.

\begin{claim}[\protect{\textnormal{cf.~\cite[Claim 2.32]{cds}}}]\label{hotmail}
Let $y\in M\setminus N$. Then there exists $x\in\Sigma$ and
$t_0>0$ such that $y=\gamma_x(t_0)$.
\end{claim}

\begin{proof}[Proof of Claim \ref{hotmail}]
Let $y\in M\setminus N$. Since $f$ is proper and bounded
from below and $X$ is complete, we know from
\cite[Proposition 2.27]{cds} that there exists a sequence
of times $t_j\to-\infty$ as $j\to\infty$ and a point
$p\in M_0(X)$ such that $\gamma_y(t_j)\longrightarrow p$
as $j\to\infty$. Now, by choice of $A$, we see that
\[
f(\gamma_y(t_j))\longrightarrow f(p)\leq A<3A
\qquad\textrm{as }j\to\infty.
\]
Consequently, there exists $t>0$ such that
$f(\gamma_y(-t))<3A<f(y)$.
By continuity, there exists $t_0\in(0,t)$ such that
$f(\gamma_y(-t_0))=3A$.
Taking $x:=\gamma_y(-t_0)\in f^{-1}(3A)=\Sigma$,
we see that $y=\Phi_{t_0}(x)=\gamma_x(t_0)$, as claimed.
\end{proof}

\noindent
For $y\in M\setminus N$, let $x\in\Sigma$ and $t_0>0$
be as in Claim \ref{hotmail}. Writing $t_0=k+s$, where
$k\in\mathbb{N}_0$ and $0\leq s<1$, we see that
\[
y=\Phi_{t_0}(x)=\Phi_k(\Phi_s(x))\in\Phi_k(K).
\]
We can therefore write
\begin{equation}\label{eq:weighted-L2-flow-cover}
M\setminus N
\subseteq\bigcup_{k=0}^{\infty}\Phi_k(K).
\end{equation}

Let $\omega$ and $d\mu$ denote the K\"ahler form and
volume form of $g$, respectively.
Set $d\mu_f:=e^{-f}d\mu$ and, for $t\geq0$, let
\[
F_t=\int_0^t f\circ\Phi_s\,ds,
\qquad
\omega_t=\Phi_t^*\omega,
\qquad
g_t=\Phi_t^*g.
\]
Since $X$ is real holomorphic and $X=\nabla^g f$, we have
$\mathcal L_X\omega=2i\partial\bar\partial f$. Thus
\[
\frac{d}{dt}\omega_t
=\Phi_t^*(\mathcal L_X\omega)
=2i\Phi_t^*(\partial\bar\partial f)
=2i\partial\bar\partial(f\circ\Phi_t).
\]
Upon integrating from $0$ to $t$, we find that
\begin{equation}\label{fuck}
\omega_t=\omega+2i\partial\bar\partial F_t.
\end{equation}

Next, note from our normalisation \eqref{normal-f} that
\begin{equation*}
\begin{split}
\mathcal L_X d\mu_f
&=\mathcal L_X(e^{-f}d\mu)=X\cdot(e^{-f})\,d\mu
  +e^{-f}\mathcal L_X d\mu=\left(-X\cdot f+\operatorname{div}_g(X)\right)
  e^{-f}d\mu\\
&=\left(\Delta_g f-X\cdot f\right)d\mu_f=-2f\,d\mu_f,
\end{split}
\end{equation*}
so that
\begin{equation*}
\begin{split}
\frac{d}{dt}\bigl(\Phi_t^*d\mu_f\bigr)=\Phi_t^*(\mathcal L_X d\mu_f)=-2\Phi_t^*(f\,d\mu_f)=-2(f\circ\Phi_t)\,\Phi_t^*d\mu_f.
\end{split}
\end{equation*}
Solving this yields
\begin{equation}\label{eq:weighted-L2-jacobian}
\Phi_t^*d\mu_f=e^{-2F_t}d\mu_f.
\end{equation}

Finally, for $y\in M$, we see from
\eqref{pressure} that
\begin{equation*}
\begin{split}
\frac{d}{dt}\bigl(f(\gamma_y(t))+n\bigr)=|\nabla^g f|_g^2(\gamma_y(t))=2\bigl(f(\gamma_y(t))+n\bigr)-R_g(\gamma_y(t))\leq2\bigl(f(\gamma_y(t))+n\bigr)
\end{split}
\end{equation*}
for all $t\in\mathbb R$ by
Theorem \ref{prop-sum-sgrs}(i).
Here $n=\operatorname{dim}_{\mathbb C}M$.
Thus, for all $t\in\mathbb R$,
\begin{equation*}
\begin{split}
\frac{d}{dt}
\left(e^{-2t}\bigl(f(\gamma_y(t))+n\bigr)\right)
&=\left(
\frac{d}{dt}f(\gamma_y(t))
-2\bigl(f(\gamma_y(t))+n\bigr)
\right)e^{-2t}\leq0.
\end{split}
\end{equation*}
Since $f+n\geq0$ by \eqref{pressure} and
Theorem \ref{prop-sum-sgrs}(i), this means that
for all $t\geq0$,
\[
0\leq f(\gamma_y(t))+n
\leq e^{2t}(f(y)+n).
\]
Integrating from $0$ to $t$, we deduce that
\begin{equation*}
\begin{split}
0\leq F_t(y)+nt=\int_0^t\bigl(f(\Phi_s(y))+n\bigr)\,ds\leq(f(y)+n)\int_0^t e^{2s}\,ds=\frac{f(y)+n}{2}(e^{2t}-1).
\end{split}
\end{equation*}
Consequently,
\[
|F_t(y)|
\leq\frac{f(y)+n}{2}(e^{2t}-1)+nt.
\]
In particular, on every fixed compact set $B\subset M$,
\begin{equation}\label{eq:weighted-L2-F-bound}
\sup_B|F_t|\leq C(B)e^{2t}
\qquad\textrm{for all }t\geq0,
\end{equation}
for a positive constant $C(B)>0$ independent of $t$.

Let $\chi\in C_c^\infty(M)$ be a compactly supported
bump function with $\chi\geq0$ and $\chi=1$ in a
neighbourhood of $K$.
Since $\omega_t$ is positive-definite, integration by
parts on the fixed compact support of $\chi$ gives us that
\begin{equation*}
\begin{split}
\int_K\left(\operatorname{tr}_\omega\omega_t\right)\omega^n&\leq
\int_M\chi\left(\operatorname{tr}_\omega\omega_t\right)\omega^n=n\int_M\chi\,\omega^n
  +2\int_M\chi\left(\Delta_\omega F_t\right)\omega^n=n\int_M\chi\,\omega^n
  +2\int_M F_t\left(\Delta_\omega\chi\right)\omega^n\\
&\leq Ce^{2t},
\end{split}
\end{equation*}
where we have used \eqref{fuck} in the first equality and
\eqref{eq:weighted-L2-F-bound} in the last line.

Now, the fact that $[X,Y]=0$ implies that $\Phi_t^*Y=Y$.
Since $g_t$ and $g$ are positive-definite and
$J$-Hermitian, we also have that
\begin{equation}\label{bollox1}
0\leq g_t(Y,Y)
\leq |Y|_g^2\cdot\left(\operatorname{tr}_\omega\omega_t\right).
\end{equation}
Moreover, since $f$ is non-decreasing along the flow of $X$,
we see that $f\circ\Phi_s\geq f\geq3A\geq2A$ on $K$ for every
$s\geq0$, so that $F_t\geq2At$ on $K$.
Using \eqref{eq:weighted-L2-jacobian} and \eqref{bollox1},
we obtain, for every integer $k\geq0$,
\begin{equation*}
\begin{split}
\int_{\Phi_k(K)}|Y|_g^2\,d\mu_f&=\int_K\Phi_k^*
  \bigl(|Y|_g^2\,d\mu_f\bigr)=\int_K\left(|Y|_g^2\circ\Phi_k\right)
  \Phi_k^*d\mu_f=\int_K\left(|Y|_g^2\circ\Phi_k\right)
  e^{-2F_k}\,d\mu_f\\
&=\int_K g_k(Y,Y)e^{-2F_k}\,d\mu_f\leq\int_K |Y|_g^2
  \left(\operatorname{tr}_\omega\omega_k\right)
  e^{-2F_k}\,d\mu_f\\
  &\leq e^{-4Ak}\int_K |Y|_g^2
\left(  \operatorname{tr}_\omega\omega_k\right)\,d\mu_f\leq C(Y,\,K)e^{-4Ak}
  \int_K\left(\operatorname{tr}_\omega\omega_k\right)\,\omega^n\\
  &\leq C(Y,\,K)e^{-4Ak}e^{2k}=C(Y,\,K)e^{-2(2A-1)k}.
\end{split}
\end{equation*}
Summing over $k$ and using
\eqref{eq:weighted-L2-flow-cover} yields a bound on
$\int_{M\setminus N}|Y|_g^2\,d\mu_f$.
The result now follows.
\end{proof}

We also know that such vector fields are complete, thanks to:
\begin{lemma}[\protect{\textnormal{cf.~\cite[Lemma 2.34]{cds}}}]\label{completee}
Let $(M,\,g,\,X)$ be a complete shrinking gradient K\"ahler-Ricci soliton. Then every real vector field $Y$ on $M$ with $[X,\,Y]=0$ is complete.
\end{lemma}

\begin{proof}
This follows from \cite[Lemma 2.34]{cds} because the zero set of $X$ is compact by \cite[Proposition 3.5]{JunshengSong}.
\end{proof}

Next, we investigate the growth of real holomorphic vector fields commuting with $X$ lying in $L^2(e^{-f}\omega^n)$, the norm of which we denote by $\|\cdot\|_{L^2_f}$.
\begin{lemma}[{\protect{\textnormal{cf.~\cite[Claim A.14]{cds}}}}]\label{sharp-growth-clim-Y}
Let $(M,\,g,\,X)$ be a complete shrinking gradient K\"ahler-Ricci
soliton of complex dimension $n$ with K\"ahler form $\omega$ and soliton vector field
$X=\nabla^{g}f$ for a smooth real-valued function
$f:M\to\mathbb{R}$. There exists $C>0$ such that every smooth real holomorphic
vector field $Y$ on $M$ with $[X,Y]=0$ satisfies
\begin{equation*}
|Y|^2_{g}\leq C\|Y\|^2_{L^2_f}(f+n+1)\quad\text{on $M$}.
\end{equation*}
\end{lemma}

\begin{proof}
We start by taking note of the following crucial equation that holds pointwise on $M$ and is independent of the soliton structure. For a real holomorphic vector field $Y$, it holds true that
\begin{equation*}
\Delta_g Y+\Ric(Y)=0.
\end{equation*}
This observation is due to Lichnerowicz and its proof can be found for instance in \cite[Proposition $4.79$]{book:Ballmann}. Since $Y$ commutes with $X$, $Y$ satisfies the following elliptic equation:
\begin{equation*}
\Delta_fY+Y:=\Delta_g Y-\nabla^g_XY+Y=\Delta_g Y-\nabla_XY+\Ric(Y)+\nabla_YX=0,
\end{equation*}
where we have used the soliton equation \eqref{krseqn} with $\lambda=1$. By Kato's inequality, it therefore follows that
\begin{equation}\label{Y-lap}
\Delta_f|Y|_g\geq -|Y|_g\quad\text{on $M$}
\end{equation}
in a weak sense.

Since $\Delta_ff=-2f$, a direct computation on the set where $\{f>0\}$ shows that
\begin{equation}\label{barrier-f-sqrt}
\Delta_f\left(f^{-\frac{1}{2}}\right)=f^{-\frac{1}{2}}+\frac{3}{4}|\nabla^gf|^2_gf^{-\frac{5}{2}}.
\end{equation}
Using \eqref{Y-lap} and \eqref{barrier-f-sqrt}, we compute that
\begin{equation}\label{finito}
\begin{split}
\Delta_{f}\left(\frac{|Y|_g}{f^{\frac{1}{2}}}\right)&\geq -\frac{|Y|_g}{f^{\frac{1}{2}}}+|Y|_g\left(f^{-\frac{1}{2}}+\frac{3}{4}|\nabla^gf|^2_gf^{-\frac{5}{2}}\right)+2g\left(\nabla^g|Y|_g,\nabla^gf^{-\frac{1}{2}}\right)\\
&\geq\frac{3}{4}|\nabla^g\ln f|^2\frac{|Y|_g}{f^{\frac{1}{2}}}-g\left(\nabla^g\left(\frac{|Y|_g}{f^{\frac{1}{2}}}\right),\nabla^g\ln f\right)\\
&\qquad-\frac{1}{2}|\nabla^g\ln f|^2\frac{|Y|_g}{f^{\frac{1}{2}}}\\
&\geq \frac{1}{4}|\nabla^g\ln f|^2\frac{|Y|_g}{f^{\frac{1}{2}}}-g\left(\nabla^g\left(\frac{|Y|_g}{f^{\frac{1}{2}}}\right),\nabla^g\ln f\right),
\end{split}
\end{equation}
a computation valid on the set where $f$ is strictly positive. 

As $u:=\frac{|Y|_g}{f^{\frac{1}{2}}}$ lies in $L^2(e^{-f+\log f}\omega^n)$ by Proposition \ref{keywest} and since the zero set of $X$ is compact \cite{JunshengSong}, a direct application of (the proof of) \cite[Lemma 4.9]{conlon321} with $V:=\frac{1}{4}|\nabla^g\ln f|^2$ shows that there exists $R>0$ such that $\{f=R\}$ is a smooth hypersurface with $\{f\leq R\}$ containing the zero set of $X$ and such that
\begin{equation*}
\sup_{\{f\,\geq\, R\}}u\leq \sup_{\{f= R\}}u.
\end{equation*}
This implies that there exists $C>0$ independent of $Y$ such that 
\begin{equation}\label{ptwise-Y}
|Y|_g\leq C(f+n+1)^{\frac{1}{2}}\sup_{\{f\,\leq\, R\}}|Y|_g
\end{equation}
pointwise on $M$. Here, $f+n+1>0$ thanks to \eqref{pressure}.

Now, as $\Delta_fY+Y=0$, a De Giorgi-Nash-Moser iteration applied to $\{f\,\leq \, 2R\}$ shows that there exists $C=C(n,g,R)>0$ satisfying $\sup_{\{f\,\leq\, R\}}|Y|^2_g\leq C\int_{\{f\,\leq\,2R\}}|Y|^2_g\,\omega^n$. Therefore, there exists a (possibly larger) positive constant $C=C(n,g,R)$ such that $\sup_{\{f\,\leq\, R\}}|Y|^2_g\leq C\|Y\|^2_{L^2_f}$. Combining this estimate with \eqref{ptwise-Y} yields the desired result.
\end{proof}

We next show that the Lie algebra $\mathfrak{aut}^{X}(M)$ of holomorphic vector fields commuting with $X$ is finite-dimensional. 

\begin{prop}[\protect{\textnormal{cf.~\cite[Proposition 5.7 \& Remark 5.8]{cds}}}]\label{finitee}
Let $(M,\,g,\,X)$ be a complete shrinking gradient K\"ahler-Ricci soliton 
and let $\mathfrak{aut}^{X}(M)$ denote the space of all real holomorphic vector fields $Y$ on $M$ with $[X,\,Y]=0$. Then $\mathfrak{aut}^{X}(M)$ is finite-dimensional.
\end{prop}

\begin{proof}
By Proposition \ref{keywest}, $\mathfrak{aut}^{X}(M)\subseteq L^2(e^{-f}\omega^n)$. By Riesz's theorem, it suffices to show that $\mathfrak{aut}^{X}(M)$ is closed (hence a Banach space) and that the closed unit ball of $\mathfrak{aut}^{X}(M)$ (with respect to the norm induced by $L^2(e^{-f}\omega^n)$) is (sequentially) compact. 

Let $(Y_i)_i$ be a bounded sequence in $\mathfrak{aut}^{X}(M)$. By Banach-Alaoglu's theorem, there exists $Y_{\infty}\in L^2(e^{-f}\omega^n)$ such that $(Y_i)_i$ converges (up to a subsequence) weakly to $Y_{\infty}\in L^2(e^{-f}\omega^n)$. 
We claim that $Y_{\infty}\in\mathfrak{aut}^{X}(M)$ and that $(Y_i)_i$ converges (up to a subsequence) strongly to $Y_{\infty}\in L^2(e^{-f}\omega^n)$. 

By Lemma \ref{sharp-growth-clim-Y}, there exists $C>0$ such that for all indices $i$, 
\begin{equation}\label{linear-growth-Y-i}
|Y_i|^2_{g}\leq C(f+n+1)\quad\text{on $M$}.
\end{equation}
Now, as $\Delta_fY_i+Y_i=0$ on $M$ (as observed in the proof of Lemma \ref{sharp-growth-clim-Y}), Schauder interior elliptic estimates show that for each $k\geq 0$ and $R>0$, there exists $C_{k,R}=C(n,k,g,R)>0$ such that $|\nabla^{g,\,k}Y_{i}|_g\leq C_{k,R}$ on $\{f\,\leq\,R\}$. By a diagonal argument and repeated applications of the Arzel\`a-Ascoli theorem, there exists a subsequence of $(Y_i)_i$, still denoted by $(Y_i)_i$, that converges locally smoothly to $Y_{\infty}$ on $M$. This proves in particular that $Y_{\infty}\in \mathfrak{aut}^{X}(M)$ and that $\mathfrak{aut}^{X}(M)$ is closed.

For a given $\varepsilon>0$, \eqref{linear-growth-Y-i} together with Theorem \ref{prop-sum-sgrs}(ii) and (iii) shows that there exists $R_{\varepsilon}=R(n,g,\varepsilon)>0$ such that 
\begin{equation}\label{tail-est}
\int_{\{f\,\geq\,R_{\varepsilon}\}}|Y_i|^2_g\,e^{-f}\omega^n\leq \varepsilon\quad\textrm{for all $i$.} 
\end{equation}
Since $(Y_i)_i$ converges locally uniformly to $Y_{\infty}$, it converges strongly to $Y_{\infty}$ in the $L^2_f$-norm on $\{f\,\leq\, R_{\varepsilon}\}$. The uniform tail estimate \eqref{tail-est} shows then that $(Y_i)_i$ converges strongly to $Y_{\infty}$ in the $L^2_f$-norm on $M$ as required since $\varepsilon>0$ is arbitrary. This shows that the closed unit ball of $\mathfrak{aut}^{X}(M)$ is compact, hence $\mathfrak{aut}^{X}(M)$ is finite-dimensional.
\end{proof}

\subsection{The automorphism groups of a shrinking K\"ahler-Ricci soliton}

As the following lemma demonstrates, the connected component of the identity of the Lie group of holomorphic isometries of a shrinking gradient K\"ahler-Ricci soliton $(M,\,J,\,g)$ that commute with the flow of $X$ is a compact Lie group. We denote this Lie group by $G_{0}^{X}$.  

\begin{lemma}[\protect{\textnormal{\cite[Lemma 2.8]{conlon321}}}]\label{class}
$G_{0}^{X}$ is a compact Lie group with respect to the compact-open topology.
\end{lemma}
\noindent In particular, the closure of the flow of $JX$ in
$G_0^X$ is a real torus $T$ acting holomorphically,
effectively, and isometrically on the shrinking soliton, with Lie algebra
$\mathfrak t$ containing $JX$. We denote the Lie algebra of $G_{0}^{X}$ by $\mathfrak{g}^{X}$. It comprises holomorphic Killing vector fields on $M$ that commute with $X$.

As explained after \cite[Proposition 5.6]{cds}, the fact that 
the Lie algebra $\mathfrak{aut}^{X}(M)$ of holomorphic vector fields commuting with $X$ is finite-dimensional by Proposition \ref{finitee} and elements of this Lie algebra are complete by 
Lemma \ref{completee} implies that 
this Lie algebra can be integrated up into a Lie group in the following way.
\begin{prop}[\protect{\textnormal{cf.~\cite[Proposition 5.6]{cds}}}]\label{eastbeach}
Let $(M,\,g,\,X)$ be a complete shrinking gradient K\"ahler-Ricci soliton.
Then there exists a unique connected Lie group $\operatorname{Aut}_{0}^{X}(M)$ (endowed with the compact-open topology) of diffeomorphisms
acting effectively on $M$ with Lie algebra $\mathfrak{aut}^{X}(M)$.
\end{prop}
\noindent $\operatorname{Aut}_{0}^{X}(M)$ is of course the connected component of the identity of the holomorphic automorphisms of $M$ that commute with the
flow of $X$.

Since $\mathfrak{g}^{X}$ is a Lie subalgebra of $\mathfrak{aut}^{X}(M)$, it is clear that $G_{0}^{X}$ is a Lie subgroup of $\operatorname{Aut}_{0}^{X}(M)$.
Being compact, $G^{X}_{0}$ is a closed subgroup of $\operatorname{Aut}_{0}^{X}(M)$, hence is a compact Lie subgroup of $\operatorname{Aut}_{0}^{X}(M)$, with everything being relative to the compact-open topology.
\begin{lemma}[\protect{\textnormal{cf.~\cite[Lemma 5.10]{cds}}}]\label{northbeach}
 $G^{X}_{0}$ is a compact Lie subgroup of $\operatorname{Aut}_{0}^{X}(M)$ (with respect to the compact-open topology).
\end{lemma}

\subsection{The weighted volume functional}\label{fat}

Let $(M,\,g,\,X)$ be a complete shrinking gradient K\"ahler-Ricci soliton of complex dimension $n$ with complex structure $J$, K\"ahler form $\omega$, and soliton vector field $X=\nabla^g f$ for $f:M\to\mathbb{R}$ smooth. Since $JX$ is complete by Lemma \ref{completee}, the closure of the flow of $JX$ in the compact Lie group $G_{0}^{X}$ generates the action of a real torus $T$ with Lie algebra $\mathfrak{t}$ acting holomorphically, effectively, and isometrically on $(M,\,g,\,J)$ with $JX\in\mathfrak{t}$. 
The Lie algebra $\mathfrak{t}$ can be identified with real holomorphic Killing vector fields on $M$. Since $M$ is simply connected
\cite{carlos, JunshengSong}, the action of $T$ is automatically Hamiltonian, and so there is a moment map for the action. Recall from \cite[Definition 5.13]{cds}
the definition of the moment map for this action.

\begin{definition}[{Moment map, \cite[Definition 5.13]{cds}}]\label{moments}
Let $\langle\cdot\,,\,\cdot\rangle$ denote the natural pairing between $\mathfrak{t}$ and $\mathfrak{t}^{*}$.
Then we define the \emph{moment map $\mu:M\to\mathfrak{t}^{*}$ for the action of $T$ on $M$} by
$$u_{Y}(x)=\langle\mu(x),\,Y\rangle,\quad x\in M,\qquad\textrm{for all $Y\in\mathfrak{t}$},$$
where $u_{Y}$ is such that $\nabla^g u_{Y}=-JY$, $\mathcal{L}_{JX}u_{Y}=0$, and $\Delta_{\omega}u_{Y}+u_{Y}+\frac{1}{2}(JY)\cdot f=0.$
\end{definition}

\noindent The fact that the said normalisation of $u_{Y}$ exists is demonstrated in \cite[Proposition 5.12]{cds}. This does not require a curvature bound.

We next define the weighted volume functional
for complete shrinking gradient K\"ahler-Ricci solitons.
\begin{definition}[{Weighted volume functional, cf.~\cite[equation (2.3)]{tianzhu2}}]\label{modfut}
Let $(M,\,g,\,X)$ be a complete shrinking gradient K\"ahler-Ricci soliton of complex dimension $n$ with complex structure $J$, K\"ahler form $\omega$, and with soliton vector field $X=\nabla^{g}f$ for a smooth real-valued function $f:M\to\mathbb{R}$, endowed with the holomorphic, effective, isometric action of a real torus $T$ with Lie algebra $\mathfrak{t}$ containing $JX$. (As explained above, one can always find such a torus.)
Let $\mu$ denote the moment map of the action as prescribed in Definition \ref{moments}.
Let $Y\in\mathfrak{t}$ and let $u_{Y}:=\langle\mu,\,Y\rangle$ be the
Hamiltonian potential of $Y$ so that
$\mathcal{L}_{JX}u_{Y}=0$ and $\Delta_{\omega}u_{Y}+u_{Y}+\frac{1}{2}(JY)\cdot f=0$. Finally, let
$$\Lambda:=\{Y\in\mathfrak{t}:\textrm{$u_{Y}$ is proper and bounded below}\}\subseteq\mathfrak{t}.$$
Then the \emph{weighted volume functional} $F$ is defined by
\begin{equation*}
F:\Lambda\longrightarrow\mathbb{R}_{>0},\qquad F(Y)=\int_{M}e^{-\langle\mu,\,Y\rangle}\omega^{n}=\int_{M}e^{-u_{Y}}\omega^{n}.
\end{equation*}
\end{definition}

The set $\Lambda$ is an open cone in $\mathfrak{t}$ which is determined by the image of $M$ under $\mu$;
see \cite[Proposition A.4]{cds}. Since $f$ is proper and bounded below, $JX$, which by assumption lies in $\mathfrak{t}$, lies in $\Lambda$ so that $\Lambda$ is non-empty.
Thus, since the zero set of $X$ is compact \cite{JunshengSong}, 
the Duistermaat-Heckman theorem (cf.~\cite[Theorem A.3]{cds}) guarantees that $F$ is well-defined.

\section{Proof of Theorem \ref{uunique}}

\subsection{A Matsushima-type theorem}

Our set-up is as in Section \ref{fat}. The goal of this section is to establish a Matsushima theorem for shrinking K\"ahler-Ricci solitons without a Ricci curvature bound.
\begin{theorem}
[\protect{\textnormal{A Matsushima theorem for shrinking K\"ahler-Ricci solitons; cf.~\cite[Theorem 5.1]{cds}}}]\label{matty}
Let $(M,\,g,\,X)$ be a complete shrinking gradient K\"ahler-Ricci soliton. Then $$\mathfrak{aut}^{X}(M)=\mathfrak{g}^{X}\oplus J\mathfrak{g}^{X}.$$
\end{theorem}

\begin{proof}
The proof proceeds as in the proof of \cite[Theorem 5.1]{cds}, the key new ingredients being Proposition \ref{keywest} 
and the fact that the zero set of $X$ is compact \cite{JunshengSong}. These we use to replace the Ricci curvature bound in the original proof. 
For completeness, we highlight the relevant changes. 

Write $\langle\cdot\,,\,\cdot\rangle$, $|\cdot|$, and $\nabla$ for the inner product, norm, and Levi-Civita connection determined by $g$ respectively, and
let $Y\in\mathfrak{aut}^{X}(M)$. Then $Y$ defines a real holomorphic vector field on $M$ with $[X,\,Y]=0$.
Take the $(1,\,0)$-part $Y^{1,\,0}$ of $Y$, i.e., let $Y^{1,\,0}=\frac{1}{2}(Y-iJY)$.
 Then $\bar{\partial}Y^{1,\,0}=0$ so that $\omega\lrcorner Y^{1,\,0}$ is a $\bar{\partial}$-closed $(0,\,1)$-form, where $\omega$ denotes the K\"ahler form of $g$. We first claim that $\omega\lrcorner Y^{1,\,0}$ admits a smooth complex potential.
\begin{claim}[\protect{\textnormal{cf.~\cite[Claim 5.2]{cds}}}]\label{baa}
There exists a smooth complex-valued function
$u_Y\in L^2(e^{-f}\omega^n)$ on $M$ such that
$-i\omega\lrcorner Y^{1,0}=\bar\partial u_Y$,
or equivalently, such that
$Y^{1,0}=\nabla^{1,0}u_Y$.
\end{claim}

\noindent Note that $u_{Y}$ is unique up to the addition of a holomorphic function.

\begin{proof}[Proof of Claim \ref{baa}]
Let $h$ denote the metric on $-K_{M}$ induced by $\omega$. Then the curvature of the metric $e^{-f}h$ on $-K_{M}$
is precisely $\omega$ by virtue of the defining equation of a shrinking gradient K\"ahler-Ricci soliton.
Treat $\omega\lrcorner Y^{1,\,0}$ as a $-K_{M}$-valued $(n,\,1)$-form. 
By Proposition \ref{keywest}, the weighted $L^2$-norm
of $Y$ is finite. Consequently, the $L^2$-norm of
$\omega\lrcorner Y^{1,0}$ measured with respect to
$e^{-f}h$ is finite. An application of H\"ormander's $L^{2}$-estimates \cite[Theorem 6.1, p.376]{demmy} now yields the desired conclusion.
\end{proof}

Next, contracting [\eqref{krseqn}, $\lambda=1$] with $Y^{1,\,0}$ and using the Bochner formula, we see that
$$-i\bar{\partial}\Delta_{\omega}u_{Y}+i\bar{\partial}(Y^{1,\,0}\cdot f)=i\bar{\partial}u_{Y}$$
so that 
\begin{equation}\label{bonjour}
\bar{\partial}(\Delta_{\omega}u_{Y}+u_{Y}-Y^{1,\,0}\cdot f)=0.
\end{equation}
By averaging $u_{Y}$ over the action of $T$, we may assume that $\mathcal{L}_{JX}u_{Y}=0$. 
Since $Y$ commutes with $JX$, this $JX$-invariant $u_{Y}$ will still be weighted-integrable and 
satisfy \eqref{bonjour}. We therefore find that 
\begin{equation*}
\Delta_{\omega}u_{Y}+u_{Y}-Y^{1,\,0}\cdot f=H
\end{equation*}
for a $JX$-invariant holomorphic function $H$ on $M$. We wish to show that $H$ is constant. 
By subtracting this constant from $u_{Y}$ and recalling that the weighted volume of $M$ is finite from Theorem \ref{prop-sum-sgrs}, we will then have normalised $u_{Y}$ such that 
\begin{equation}\label{start3}
\Delta_{\omega}u_{Y}+u_{Y}-Y^{1,\,0}\cdot f=0,
\end{equation}
while still preserving the fact that $u_{Y}\in L^2(e^{-f}\omega^n)$.

To this end,  let $X^{0,1}:=\frac12(X+iJX)$.
Since $H$ is holomorphic, we see that
\begin{equation*}
\begin{split}
0=\bar\partial H(X^{0,1})=X^{0,1}\cdot H=\frac12\left(X\cdot H+iJX\cdot H\right).
\end{split}
\end{equation*}
Using the $JX$-invariance of $H$, it immediately follows that $X\cdot H=0$, i.e., $H$ is constant along the flow-lines of $X$.
Let $K$ and $N$ be the compact subsets of $M$ as defined in the proof of Proposition \ref{keywest}. By Claim \ref{hotmail},
every positive flow-line of $-X$ starting at a point in $M\setminus N$ intersects $K$. By what we have just said, 
we find that $\sup_{M}|H|=\sup_{N\cup K}|H|$. By the maximum principle, it now follows that $H$ is constant, as required. 

Before continuing, we need to establish the weighted $L^2$-integrability of the total gradient and second covariant derivatives of $u_{Y}$.
\begin{claim}[\protect{\textnormal{cf.~\cite[Claim 5.4]{cds}}}]\label{ronanan}
The gradient $\nabla u_Y$ and the second covariant derivatives $\nabla^2 u_Y$ of $u_Y$ belong to $L^2(e^{-f}\omega^n)$.
\end{claim}

\begin{proof}[Proof of Claim \ref{ronanan}]
Since $\Delta_{\omega}u_Y=Y^{1,\,0}\cdot f-u_Y$ by \eqref{start3}, we see that the drift Laplacian of $u_{Y}$ satisfies
\begin{equation*}
\Delta_{\omega,\,X}u_Y:=\Delta_{\omega}u_Y-\frac{X}{2}\cdot u_Y=Y^{1,\,0}\cdot f-u_Y-\frac{X}{2}\cdot u_Y=-u_{Y}.
\end{equation*}
Here we have used in the last equality the fact that
\[
Y^{1,0}\cdot f
=(\nabla^{1,\,0}u_Y)\cdot f
=X^{0,\,1}\cdot u_Y
=\frac{X}{2}\cdot u_Y,
\]
because $\mathcal L_{JX}u_Y=0$. Since $u_{Y}\in L^2(e^{-f}\omega^n)$, it follows that $\Delta_{\omega,\,X}u_Y$ lies in the same space.
The remainder of the proof follows as in the proof of
\cite[Claim 5.4]{cds}; no curvature bound is required.
\end{proof}
The remainder of the proof, including \cite[Claim 5.5]{cds}, now follows that of \cite[Theorem 5.1]{cds} verbatim.
\end{proof}

By Lemma \ref{northbeach}, we know that $G^{X}_{0}$ is a compact Lie subgroup of $\operatorname{Aut}_{0}^{X}(M)$ (with respect to the compact-open topology).
Thanks to Theorem \ref{matty}, we can now ascertain that it is a maximal compact Lie subgroup of $\operatorname{Aut}_{0}^{X}(M)$. The proof follows verbatim 
the proof of \cite[Corollary 5.11]{cds}.
\begin{corollary}[\protect{\textnormal{cf.~\cite[Corollary 5.11]{cds}}}]\label{maxx}
Let $(M,\,g,\,X)$ be a complete shrinking gradient K\"ahler-Ricci soliton.
Then $G^{X}_{0}$ is a maximal compact Lie subgroup of $\operatorname{Aut}_{0}^{X}(M)$.
\end{corollary}

\subsection{A volume minimising principle}

We next explore some elementary properties of the previously defined weighted volume functional $F$,
in particular the desired property that characterises $JX\in\mathfrak{t}$ as the unique critical point of $F$. It is here where our normalisation of the Hamiltonian potentials comes into play.
\begin{lemma}[\protect{\textnormal{Volume minimising principle; cf.~\cite[Lemma 5.16]{cds}}}]\label{critical}
Let $(M,\,g,\,X)$ be a complete shrinking gradient K\"ahler-Ricci soliton of complex dimension $n$ with soliton vector field $X=\nabla^{g}f$ for a smooth real-valued function $f:M\to\mathbb{R}$, endowed with the holomorphic, effective, isometric action of
a real torus $T$ with Lie algebra $\mathfrak{t}$ containing $JX$. Then:
\begin{enumerate}[label=\textnormal{(\roman{*})}, ref=(\roman{*})]
\item $F$ is strictly convex on $\Lambda$.
\item $JX$ is the unique critical point of $F$ in $\Lambda$.
\end{enumerate}
\end{lemma}

\begin{remark}
Part (ii) of this lemma was proved under the assumption of bounded Ricci curvature in \cite[Lemma 5.16(ii)]{cds}.
\end{remark}

\begin{proof}[Proof of Lemma \ref{critical}]
\begin{enumerate}[label=\textnormal{(\roman{*})}, ref=(\roman{*})]
\item This follows from the proof of \cite[Lemma 5.16(i)]{cds}.
\item Let $\omega$ denote the K\"ahler form of $g$ and let $Y\in \Lambda$. Let $u_{Y}:M\to\mathbb{R}$ denote the Hamiltonian potential of $Y$. Then $u_{Y}$ is proper and bounded from below. It suffices to show that 
\begin{equation*}
\int_Me^{-u_Y}\,\omega^n\geq \int_Me^{-f}\,\omega^n\quad\textrm{with equality if and only if $u_Y=f$.}
\end{equation*}

(Note that both integrals are finite by the Duistermaat-Heckman formula \cite[Theorem A.3]{cds} because the zero set of $JX$ is compact \cite{JunshengSong}.)
To this end, observe that since $e^{x}\geq1+x$ for all $x\in\mathbb{R}$, we have that $e^{-u_Y}\geq e^{-f}-(u_Y-f)e^{-f}$ on $M$, so all we need show in fact is that 
$u_Y,\,f\in L^1(e^{-f}\omega^n)$ and that $\int_M(u_Y-f)\,e^{-f}\omega^n\leq 0$.

Now, as $\Delta_{\omega,\,X}f+f=0$,
an integration by parts shows that $\int_Mf\,e^{-f}\omega^n=0$. As for the integrability of $u_Y$, let $R>1$ and set $\varphi_R:=\psi\left(\frac{u_{Y}}{R}\right)$, where $\psi:\R\rightarrow[0,\,1]$ is a smooth non-increasing function such that $\psi\equiv 1$ on $(-\infty,1]$ and $\psi\equiv 0$ on $[2,+\infty)$. Using the properties of $u_{Y}$, we see that $\varphi_{R}=1$ on the compact set $\{u_{Y}\leq R\}$, and is supported on the compact set $\{u_{Y}\leq 2R\}$. As $R\to\infty$, its support will cover all of $M$ because $u_{Y}\to\infty$ by properness.
With this in mind, the normalisation for $u_{Y}$ gives us that
\begin{equation*}
\begin{split}
\int_M u_Y\varphi_R\,e^{-f}\omega^n
&=-\int_M(\Delta_{\omega,X}u_Y)\varphi_R
    \,e^{-f}\omega^n\\
&=\frac12\int_M
    \langle\nabla^g u_Y,\nabla^g\varphi_R\rangle_g
    \,e^{-f}\omega^n\\
&=-\frac12\int_M(JY)\cdot\varphi_R
    \,e^{-f}\omega^n\\
&=\frac{1}{2R}\int_M|Y|_g^2
    \psi'\left(\frac{u_Y}{R}\right)e^{-f}\omega^n\\
&\leq0.
\end{split}
\end{equation*}
Letting $R\to\infty$, we deduce that not only is $u_Y \in L^1(e^{-f}\omega^n)$ (and consequently so is  $\Delta_{\omega,\,X}u_Y$), but also that $\int_Mu_Y\,e^{-f}\omega^n\leq 0$, which, together with the strict convexity given by part (i), leads to the desired conclusion.
\end{enumerate}
\end{proof}

\subsection{Proof of Theorem \ref{uunique}}

We next apply the above theory to prove the uniqueness of
the soliton vector field, thereby proving Theorem \ref{uunique}.
Our proof follows along the lines of that of
\cite[Theorem 5.18]{cds}, with the extra input of
\cite[Proposition 5.17]{Esp25}.

Given two complete shrinking gradient K\"ahler-Ricci solitons
$(M,\,g_i,\,X_i)$, $i=1,\,2$, on $M$, consider the flow of
the real Killing vector field $\xi_i:=JX_i$ in $G_0^{X_i}$.
Since $G_0^{X_i}$ is a compact Lie group by
Lemma \ref{class}, the closure of the flow of $\xi_i$
in $G_0^{X_i}$ generates the effective holomorphic
isometric action of a real torus $T_i\subseteq G_0^{X_i}$
with $\xi_i$ in the Lie algebra of $T_i$.
In \cite{JunshengSong}, it is shown that there exists
a map $\pi_i:M\to Y_i$ with $Y_i$ an affine cone such that
$(\pi_i:M\to Y_i,\,\xi_i,\,T_i)$ defines a polarised
Fano fibration. In particular, by definition, the induced
action of $T_i$ on $Y_i$ has a unique fixed point $x_i$.

Now, $T_i$ lives in $G_0^i$, the connected component
of the identity of the Lie group of holomorphic isometries
of $(M,\,J,\,g_i)$.
Take a maximal torus $T_i'$ in $G_0^i$ containing $T_i$.
Then we claim that
\linebreak $(\pi_i:M\to Y_i,\,\xi_i,\,T_i')$ still defines a
polarised Fano fibration.
This is explained above \cite[Proposition 5.17]{Esp25}, but we check it here directly.
All that needs to be shown is that 
the induced action of $T'_{i}$ on $Y_{i}$ has an isolated fixed point which is $x_{i}$. To this end, 
note first that as a larger torus, the
fixed-point set $\operatorname{Fix}_{Y_i}(T_i')$
of $T_i'$ on $Y_i$ must be contained in the fixed-point
set $\operatorname{Fix}_{Y_i}(T_i)$ of $T_i$.
To show the reverse inclusion, we need to show that
\[
\operatorname{Fix}_{Y_i}(T_i)
\subseteq\operatorname{Fix}_{Y_i}(T_i').
\]
Now, $\operatorname{Fix}_{Y_i}(T_i)$ contains only $x_i$,
and we have that $t\cdot x_i=x_i$ for all $t\in T_i$.
For $t'\in T_i'$, we then have that
\[
t'\cdot(t\cdot x_i)=t'\cdot x_i,
\]
i.e., since $T_i'$ is abelian,
\[
t\cdot(t'\cdot x_i)=t'\cdot x_i.
\]
That is to say, $t'\cdot x_i$ is fixed by $T_i$,
and so because $\operatorname{Fix}_{Y_i}(T_i)$
comprises only $x_i$, it follows that
$t'\cdot x_i=x_i$ and $x_i$ is the unique fixed point
of the action of $T_i'$ on $Y_i$.

We next claim that $T_{i}'$ is maximal in $\operatorname{Aut}(M)$.
Indeed, since $T_i'$ contains $T_i$, it commutes with
$\xi_i=JX_i$ and hence with $X_i$.
Thus, $T_i'\subseteq G_0^{X_i}$ and is a maximal torus in this group since $G_{0}^{i}\supseteq G_{0}^{X_{i}}$.
By Corollary \ref{maxx}, $T_i'$ is maximal in
$\operatorname{Aut}_0^{X_i}(M)$.
Now, any holomorphic torus in $\operatorname{Aut}(M)$ containing $T_i'$ also commutes
with $X_i$, and so lies in $\operatorname{Aut}_0^{X_i}(M)$.
Consequently, $T_i'$ is maximal in $\operatorname{Aut}(M)$, as claimed.
Since $(\pi_i:M\to Y_i,\,\xi_i,\,T_i')$ defines a polarised Fano fibration,
$T'_1$ and $T'_2$ are conjugate
by \cite[Proposition 5.17]{Esp25}, and so after pulling back by a biholomorphism, we may 
assume that there exists a maximal torus $\mathbb T$
in $\operatorname{Aut}(M)$ lying in
$G_0^1\cap G_0^2$ with both
$\xi_1,\,\xi_2\in\operatorname{Lie}(\mathbb T)$.

The rest of the proof now proceeds as in
\cite[Theorem 5.18]{cds}, using the compactness of the
zero sets of $X_i$ provided by \cite{JunshengSong}
and Lemma \ref{critical}.

\bibliographystyle{amsalpha}
\bibliography{ref2}

\end{document}